\documentclass{amsart}[12pt]
\usepackage{amsmath, amsthm, amscd, amsfonts,}
\usepackage{graphicx,float}

\usepackage{amssymb}

\usepackage{pb-diagram}

\newcommand{\PP}{{\mathbb{P}}}

\newcommand{\url}[1]{{\tt #1}}

\newcommand{\gl}{\lambda}

\DeclareMathOperator{\dom}{dom}

\DeclareMathOperator{\Suc}{Suc}
\DeclareMathOperator{\Lev}{Lev}

\def\MQB{{\mathbb{Q}}}

\def\k{\kappa}

\def\l{\lambda}

\def\a{\alpha}

\newtheorem{theorem}{Theorem}[section]
\newtheorem{lemma}[theorem]{Lemma}

\newtheorem{definition}[theorem]{Definition}
\newtheorem{open Question}[theorem]{Open Question}

\newtheorem{remark}[theorem]{Remark}

\numberwithin{equation}{section}

\def\MQB{{\mathbb{Q}}}

\def\k{\kappa}

\def\l{\lambda}

\def\a{\alpha}

\def\l{\lambda}

\def\rmark{\mbox{$\rm\bf\rule{0.06em}{1.45ex}\kern-0.05em R$}}
\def\pmark{\mbox{$\rm\bf\rule{0.06em}{1.45ex}\kern-0.05em P$}}
\def\nmark{\mbox{$\rm\bf\rule{0.06em}{1.45ex}\kern-0.05em N$}}
\def\vdash{\mbox{$\rm\| \kern-0.13em -$}}

\usepackage{pb-diagram}

\def\l{\lambda}

\def\rmark{\mbox{$\rm\bf\rule{0.06em}{1.45ex}\kern-0.05em R$}}
\def\pmark{\mbox{$\rm\bf\rule{0.06em}{1.45ex}\kern-0.05em P$}}
\def\nmark{\mbox{$\rm\bf\rule{0.06em}{1.45ex}\kern-0.05em N$}}
\def\vdash{\mbox{$\rm\| \kern-0.13em -$}}

\begin{document}

\title[All uncountable regular cardinals can be inaccessible in $\text{HOD}$]{All uncountable regular cardinals can be inaccessible in $\text{HOD}$ $^{1}$}
\author[ M. Golshani.]{Mohammad Golshani}

\thanks{The  author's research has been supported by a grant from IPM (No. 91030417).}

\thanks{The result of this paper is motivated by a suggestion of Moti Gitik, to whom the author is very thankful.}

\thanks{$^{1}$ A strengthening of the result of this paper is proved by Gitik-Merimovich \cite{gitik-merimovich}, where they produced a model in which all uncountable regular cardinals are measurable in $\text{HOD}$.
}
\maketitle

\begin{abstract}
Assuming the existence of a supercompact cardinal and an inaccessible above it, we construct a model of $\text{ZFC}$, in which all uncountable regular cardinals are inaccessible in $\text{HOD}$.
\end{abstract}

\section{introduction}
An important development in large cardinal theory is the construction of inner models $\text{M}$ all of whose sets are definable from ordinals
and which serve as good approximations  to the entire universe $\text{V}$. The former means that $\text{M}$ is contained in $\text{HOD}$, the universe of
hereditarily ordinal definable sets, and the latter can be interpreted in a number of ways:

\begin{enumerate}
\item [$(\alpha)$] $\text{V}$ covers $\text{M}$; in the sense that every uncountable set of ordinals in $\text{V}$
is covered by a set of ordinals in $\text{M}$ of the same $\text{V}$-cardinality.

\item [$(\beta)$] $\text{V}$ weakly covers $\text{M}$; in the sense that $\alpha^+$ of $\text{M}$  equals $\alpha^+$ of $\text{V}$ for every
singular cardinals
$\alpha$ of $\text{V}$. This is for example the case if $\text{V}$ does not contain $0^{\sharp}$ and $\text{M}$ equals $\text{L}$ \cite{devlin-jensen}, or if $\text{V}$ does not
contain an inner model with a Woodin cardinal and $\text{M}$ is the core model $\text{K}$ for a Woodin cardinal \cite{jensen-steel}.
\end{enumerate}
It is easily seen that if $\text{V}$ covers $\text{M}$, then $\text{V}$ weakly covers $\text{M}$.
In \cite{c-f-g}, it is shown that we can force,  in a certain sense, the ultimate failure of weak
covering:
\begin{theorem} (\cite{c-f-g})
Suppose $\text{GCH}$ holds and $\kappa$ is a $\kappa^{+3}$-supercompact cardinal. Then there is a generic extension $\text{W}$ of $\text{V}$ in which
$\kappa$ remains inaccessible and for all infinite cardinals $\alpha <\kappa$, $(\alpha^{+})^{\text{HOD}}<\alpha^{+}$.
 In particular the rank-initial
segment $W_\k$ is a model of $\text{ZFC}$ in which for all infinite cardinals $\alpha$, $(\alpha^{+})^{\text{HOD}}<\alpha^{+}$.
\end{theorem}
The problem of finding a model of $\text{ZFC}$ in which all uncountable regular cardinals are inaccessible in $\text{HOD}$ remained open in \cite{c-f-g}.

On the other hand $\text{HOD}$ plays an important role in Woodin's recent work. The following is an important result in these directions:
\begin{theorem}
(The $\text{HOD}$ Dichotomy theorem) Assume that $\delta$ is an extendible cardinal. Then exactly one of
the following holds:
\begin{enumerate}
\item For every singular cardinal $\gamma > \delta,$ $\gamma$
 is singular in $\text{HOD}$ and $(\gamma^+)^{\text{HOD}}=\gamma^+,$

\item Every regular cardinal greater than $\delta$ is measurable in $\text{HOD}$.
\end{enumerate}
\end{theorem}
A cardinal  $\kappa$ is called $\omega$-strongly measurable in $\text{HOD}$ if there exists $\lambda < \kappa$ such that $(2^{\lambda})^{\text{HOD}} < \kappa$ and such that
there is no partition of $S = \{\alpha < \kappa: \text{cf}(\alpha) = \omega\}$ into $\lambda$ many sets $\langle S_{\alpha}: \alpha < \lambda\rangle \in \text{HOD}$ such
each set $S_{\alpha}$ is stationary in $\text{V}$.
One of the major open problems related to Woodin's work in the following, which is known as $\text{HOD}$ conjecture.
\\
{\bf The $\text{HOD}$ conjecture.} There is a proper class of regular cardinals that are not $\omega$-strongly measurable in $\text{HOD}$.

We refer to \cite{woodin1} for  more information about Woodin's work.
It turns out that if $\delta$ is an extendible cardinal, then the $\text{HOD}$ Conjecture
is equivalent to the failure of clause $(2)$ of the dichotomy  theorem 1.2. Also note that if $\text{HOD}$ is correct about singular cardinals and computes their
successors correctly, then the $\text{HOD}$ Conjecture
holds, as
\begin{center}
$\{ \gamma^+: \gamma$ is a singular cardinal $ \}$
\end{center}
is a proper class of regular cardinals which are not $\omega$-strongly measurable in
$\text{HOD}.$
On the other hand, in his talk \cite{woodin} presented at the Bristol University,  Woodin has introduced the following version of  $\text{HOD}$ conjecture:
\\
{\bf The $\text{HOD}$ conjecture (strong version).} There is a proper class of uncountable regular cardinals $\k$ which
are not measurable cardinals in $\text{HOD}$.

We may note that if $\k$ is  $\omega$-strongly measurable in
$\text{HOD},$ then it is measurable in $\text{HOD}$, and hence strong $\text{HOD}$ conjecture implies $\text{HOD}$ conjecture.
In this paper, we address these problems, and prove the following theorem
\begin{theorem}
Suppose $\text{GCH}$ holds, $\kappa$ is a supercompact cardinal and $\lambda> \k$ is inaccessible. Then there is a generic extension $\text{W}$ of $\text{V}$ in which
$\kappa$ remains inaccessible and for all infinite cardinals $\alpha <\kappa$, $\alpha^{+}$ is inaccessible in $\text{HOD}$.
 In particular the rank-initial
segment $\text{W}_\l,$ where $\l\leq \k$ is the least inaccessible cardinal of $\text{W}$ is a model of $\text{ZFC}$ in which all uncountable regular cardinals $\alpha$, are inaccessible in $\text{HOD}$.
\end{theorem}
The theorem answers the question left open in \cite{c-f-g}, and proves the consistency of the negation of a weak version of Woodin's strong $\text{HOD}$ conjecture, where measurable cardinals are replaced by inaccessible cardinals.

We may mention that a stronger result, which completely answer Woodin's strong $\text{HOD}$ conjecture was proved recently by Gitik-Merimovich \cite{gitik-merimovich};
however our proof of Theorem 1.3 has differences with the one given in \cite{gitik-merimovich}, though the main forcing notion used in both of the proofs, i.e., supercompact extender based Radin forcing, is the same.

The structure of the paper is as follows. In section 2, we give some preliminaries and results which appear later in our work. In section 3, we consider a simpler problem, namely the consistency of the existence of  a singular cardinal $\k$ such that $\k^+$ is inaccessible in $\text{HOD}$. We give the consistency of this problem in some details,  because it gives some motivations for the proof of the main theorem, whose proof is much more complicated. Finally in section 4, we complete the proof of the above mentioned theorem.

\section{Some preliminaries}
In this section, we present some definitions and results which appear in next sections.
Let's start with the definition of a projection map between forcing notions.
\begin{definition}
Let $\PP, \MQB$ be two forcing notions. $\pi$ is a projection from $\PP$ into $\MQB$ if $\pi: \PP \rightarrow \MQB,$ and it satisfies the following conditions:

$(1)$ $\pi(1_\PP)=1_{\MQB},$

$(2)$ $\pi$ is order preserving; i.e., $p \leq_\PP q \Rightarrow \pi(p) \leq_\MQB \pi(q),$

$(3)$ If $p\in \PP, q\in \MQB$ and $q \leq_\MQB \pi(p)$, then there exists $p^* \leq_\PP p$ such that $\pi(p^*) \leq_\MQB q.$
\end{definition}
It is clear that if $\pi: \PP \rightarrow \MQB$ is a projection from $\PP$ into $\MQB$, then $\pi[\PP]$ is dense in $\MQB.$
The next lemma shows that if $\PP$ projects into $\MQB,$ then a generic filter for $\PP$ yields  a generic filter for $\MQB.$
\begin{lemma}
Let $\pi: \PP \rightarrow \MQB$ be a projection from $\PP$ into $\MQB$,  let $\text{G}$ be $\PP$-generic over $\text{V}$, and let $\text{H} \subseteq \MQB$ be the filter
generated by $\pi[\text{G}].$ Then $\text{H}$ is $\MQB$-generic over $\text{V}$ and $\text{V[H]} \subseteq \text{V[G]}.$
\end{lemma}
Prikry type forcing notions arise in our work.
\begin{definition}
$\langle \PP, \leq, \leq^* \rangle$ is of Prikry type, iff

$(1)$ $\leq^* \subseteq \leq,$

$(2)$ For any $p\in \PP$ and any statement $\sigma$ in the forcing language $\langle \PP, \leq \rangle$, there exists $q \leq^* p$

$\hspace{0.5cm}$ which decides $\sigma.$
\end{definition}
The relation $\leq^*$ is usually called the Prikry relation.
The following is well-known.
\begin{lemma}
Assume $\langle \PP, \leq, \leq^* \rangle$ is of Prikry type, and suppose $\langle \PP, \leq^* \rangle$ is $\k$-closed, where $\k$ is regular uncountable. Then Forcing with $\langle \PP, \leq \rangle$ does not add new bounded subsets to $\k.$
\end{lemma}
Projection between Prikry type forcing notions arises in our work in several places. So let's present a new definition, and give an application of it.
\begin{definition}
Let $\langle \PP, \leq_\PP, \leq^*_\PP \rangle$ and $\langle \MQB, \leq_\MQB, \leq^*_\MQB \rangle$ be two forcing notions with $\leq^*_\PP \subseteq \leq_\PP$ and $\leq^*_\MQB \subseteq \leq_\MQB$. A map $\pi: \PP \rightarrow \MQB$ is a ``projection of Prikry type'' iff

$(1)$ $\pi$ is a projection from $\langle \PP, \leq_\PP \rangle$ into $\langle \MQB, \leq_\MQB \rangle$,

$(2)$ $\pi$ preserves the $\leq^*$-relation, i.e.,  $p \leq^*_\PP q \Rightarrow \pi(p) \leq^*_\MQB \pi(q),$

$(3)$ If $p\in \PP, q\in \MQB$ and $q \leq_\MQB \pi(p)$, then there exists $p^* \leq_\PP p$ such that $\pi(p^*) \leq^*_\MQB q.$
\end{definition}
It is clear that if $\pi: \PP \rightarrow \MQB$ is a projection of Prikry type from $\PP$ into $\MQB$, then $\pi[\PP]$ is dense in $\MQB,$ with respect to both $\leq$ and $\leq^*$ relations.
Note that in the above definition we did not require $\langle \PP, \leq_\PP, \leq^*_\PP \rangle$ and $\langle \MQB, \leq_\MQB, \leq^*_\MQB \rangle$ be Prikry type forcing notions.
The following lemma shows the importance of Prikry type projections.
\begin{lemma}
Assume  $\pi: \PP \rightarrow \MQB$  is a projection of Prikry type, and assume $\langle \PP, \leq_\PP, \leq^*_\PP \rangle$ satisfies the Prikry property. Then $\langle \MQB, \leq_\MQB, \leq^*_\MQB \rangle$ is also of Prikry type.
\end{lemma}
\begin{proof}
First we show that $\langle \pi[\PP], \leq_\MQB, \leq^*_\MQB \rangle$ satisfies the Prikry property.
Let $b\in R.O(\pi[\PP])$ and $q \in \pi[\PP].$ Let $p \in \PP$ be such that $q=\pi(p).$ Then  there is $p^* \leq^*_\PP p$ such that $p^*$ decides $\| b\in \pi[\dot{G}]\|_{R.O(\pi[\PP])}$ where $\dot{G}$ is the canonical name for a generic filter over $\PP.$ Let $q^*=\pi(p^*).$ Then $q^*\leq^*_\MQB q$ and it decides $b$.

But  $\pi[\PP]$ is in fact $\leq^*_\MQB-$dense in $\MQB,$ and hence
$\langle \MQB, \leq_\MQB, \leq^*_\MQB \rangle$ satisfies the Prikry property.
\end{proof}

\section{The consistency of $\k$ is singular and $\k^+$ is inaccessible in $\text{HOD}$}
In this section, we prove the following result, which is a very weak version of our main theorem.
We have decided to bring it, because it   motivates the main construction of the paper, without going into many complications we arrive in the proof of main theorem.

\begin{theorem}
Assume $\text{GCH}$ holds, $\k$ is a supercompact cardinal and $\l>\k$ is inaccessible. Then there is a generic extension $\text{V[G]}$ of the universe $\text{V}$, in which $\k$ is singular of cofinality $\omega, (\k^+)^{\text{V[G]}}=\l,$ and $\lambda$ is strongly inaccessible in $\text{HOD}^{\text{V[G]}}.$
\end{theorem}
The rest of this section is devoted to the proof of the above theorem.
We assume  the following hold in $V$:
\begin{enumerate}
\item $\text{GCH},$
\item $j: \text{V} \rightarrow \text{M}$ is an elementary embedding with $crit(j)=\k$ and $\text{M} \supseteq$ $^{<\lambda}\text{M},$
\item $\lambda>\k$ is strongly inaccessible.
\end{enumerate}
In  subsection 3.1, we use  Merimovich's ``supercompact extender based Prikry forcing'' $\PP_{E}$ \cite{mer4}, to find a generic extension $V[G]$, which  makes $\k$ singular of cofinality $\omega$, and collapses all cardinals in $(\k, \lambda)$ into $\k,$ while preserves $\lambda,$ so that $(\k^+)^{\text{V[G]}}=\lambda.$
Then in subsection 3.2, we present a cardinal preserving  forcing notion $\PP^\pi_{E}$, that we call it ``projected supercompact extender based Prikry forcing'', where in its generic extension, $\k$ becomes singular of countable cofinality and $\lambda>\k$ remains inaccessible, and next we show that there is a projection from  $\PP_{E}$ into  $\PP^\pi_{E}$, so that we can find a generic $G^\pi$ for $\PP^\pi_{E}$ with $\text{V}[\text{G}^\pi] \subseteq \text{V[G]}$. In subsection 3.5 we show that the resulting quotient forcing has enough homogeneity properties so that $\text{HOD}^{\text{V[G]}} \subseteq \text{V}[\text{G}^\pi],$ and from it we get the result.

\subsection{Supercompact extender based Prikry forcing}
In this subsection, we present Merimovich's ``supercompact extender based Prikry forcing''.
For each $\a<j(\gl)$ let
$\lambda_\a$ be minimal $\eta<\gl$ such that $\a < j(\eta),$
and let $E(\a) \subseteq P(\lambda)$ be defined by
\begin{center}
$A\in E(\a) \Leftrightarrow \a \in j(A).$
\end{center}
Note that each $E(\a)$ is a $\k$-complete ultrafilter on $\lambda$ and it has concentrated on $\lambda_\a$. Also let
\begin{center}
$i_\a: \text{V} \rightarrow \text{N}_\a \simeq \text{Ult(V}, E(\a)).$
\end{center}
Finally let
\begin{center}
$E=  \langle \langle E(\a): \a<j(\lambda)    \rangle, \langle  \pi_{\beta, \a}: \beta, \a< j(\gl), \a\in rnge(i_\beta)          \rangle \rangle$
\end{center}
be the extender derived from $j$, where $\pi_{\beta, \a}: \lambda \rightarrow \lambda$ is such that $j(\pi_{\beta, \a})(\beta)=\a$ (such a $\pi_{\beta, \a}$ exists as $\a \in rnge(i_\beta)$). Let $i: \text{V} \rightarrow \text{N }\simeq \text{Ult(V, E)}$ be the resulting extender embedding. We may assume that $j=i.$
\begin{definition}
Let $d\in [j(\lambda)]^{<\lambda}$ be such that $\k, |d| \in d.$ Then $\nu \in \text{OB(d)}$ iff:
\begin{enumerate}
\item $\nu: \dom(\nu) \rightarrow \lambda,$ where $\dom(\nu) \subseteq d,$
\item $\k, |d| \in \dom(\nu),$
\item $|\nu| \leq \nu(|d|),$
\item $\forall \a<\l$ $(j(\a)\in \dom(\nu) \Rightarrow \nu(j(\a))=\a ),$
\item $\a\in \dom(\nu) \Rightarrow \nu(\a) < \gl_\a,$
\item $\a < \beta$ in $\dom(\nu) \Rightarrow \nu(\a) < \nu(\beta).$
\end{enumerate}
Also for $\nu_0, \nu_1 \in \text{OB(d)},$ set $\nu_0 < \nu_1$ iff
\begin{enumerate}
\item [(6)]$\dom(\nu_0) \subseteq \dom(\nu_1),$

\item [(7)] For all $\a\in \dom(\nu_0)\setminus j[\gl], \nu_0(\a) < \nu_1(\a).$
\end{enumerate}
\end{definition}
We now define the forcing notion $\PP^*_E.$
\begin{definition}
$\PP^*_E$ consists of all functions $f: d \rightarrow \lambda^{<\omega}$, where  $d\in [j(\lambda)]^{<\lambda}$, $\k, |d| \in d,$ and  such that

$(1)$ For any $j(\a) \in d, f(j(\a))=\langle \a \rangle,$

$(2)$ For any $\a\in d\setminus j[\gl], $ there is some $k<\omega$ such that
\begin{center}
$f(\a)= \langle  f_0(\a), \dots, f_{k-1}(\a)               \rangle \subseteq \gl_\a $
\end{center}
$\hspace{0.7cm}$ is a finite increasing subsequence of $\gl_\a.$ For $f, g\in \PP^*_E,$
\begin{center}
$f \leq^*_{\PP^*_E} g \Leftrightarrow f \supseteq g.$
\end{center}
\end{definition}
\begin{remark}
$\langle \PP^*_E, \leq^*_{\PP^*_E} \rangle \approx Add(\lambda, |j(\lambda)|).$
\end{remark}
\begin{definition}
Assume $d\in [j(\lambda)]^{<\lambda}$ and $\k, |d| \in d.$
 Let $T \subseteq OB(d)^{<\xi} (1<\xi \leq \omega)$ and  $n<\omega.$ Then

 $(1)$ $Lev_n(T)=T \cap \text{OB(d)}^{n+1},$

 $(2)$ $\Suc_T(\langle \rangle) = \Lev_0(T),$

 $(3)$ $\Suc_T(\langle \nu_o, \dots, \nu_{n-1}         \rangle)=\{\mu\in OB(d): \langle \nu_o, \dots, \nu_{n-1}, \mu \rangle \in T    \}.$
\end{definition}
\begin{definition}
Assume $d\in [j(\lambda)]^{<\lambda}$ and $\k, |d| \in d.$
Let $T \subseteq OB(d)^{<\xi} (1<\xi \leq \omega)$. For $\langle \nu \rangle \in T,$ let
\begin{center}
$T_{\langle \nu \rangle}=\{  \langle \nu_o, \dots, \nu_{k-1} \rangle: k<\omega, \langle \nu, \nu_o, \dots, \nu_{k-1} \rangle \in T \}$
\end{center}
 and define by recursion for $\langle \nu_o, \dots, \nu_{n-1} \rangle \in T,$
\begin{center}
$T_{\langle \nu_o, \dots, \nu_{n-1} \rangle}= (T_{\langle \nu_o, \dots, \nu_{n-2} \rangle})_{\langle \nu_{n-1} \rangle}.$
\end{center}
\end{definition}
\begin{definition}
Assume $d\in [j(\lambda)]^{<\lambda}$ and $\k, |d| \in d.$
We define the measure $E(d)$ on $\text{OB(d)}$ by
\begin{center}
 $E(d)=\{ X \subseteq \text{OB(d)}: \text{mc(d)}\in j(X) \},$
\end{center}
where $\text{mc(d)}=\{\langle j(\a), \a \rangle : \a\in d   \}.$
\end{definition}
\begin{definition}
Assume $d\in [j(\lambda)]^{<\lambda}$ and $\k, |d| \in  d.$ Let $T \subseteq \text{OB(d)}^{<\omega}$ be a tree.
 $T$ is called an $E(d)$-tree, if
\begin{enumerate}
\item $\forall \langle \nu_0, \dots, \nu_{n-1} \rangle \in T$ $(\nu_0 < \dots < \nu_{n-1}),$
\item $\forall \langle \nu_0, \dots, \nu_{n-1} \rangle \in T$ $(\Suc_T(\langle \nu_0, \dots, \nu_{n-1} \rangle)\in E(d)).$
\end{enumerate}
\end{definition}
\begin{definition}
Assume $d\in [j(\lambda)]^{<\lambda}$ and  $A \subseteq \text{OB(d)}^{<\omega}.$ Then
\begin{center}
 $A \upharpoonright c =\{\langle \nu_0 \upharpoonright c, \dots \nu_{n-1} \upharpoonright c \rangle: n< \omega, \langle \nu_0, \dots, \nu_{n-1} \rangle \in A \}.$
 \end{center}
\end{definition}
\begin{remark}
For $f\in \PP^*_E,$ we use $\text{OB(f)}, E(f)$ and $\text{mc(f})$ to denote $\text{OB}(\dom(f)), E(\dom(f))$ and $\text{mc}(\dom(f))$ respectively.
\end{remark}
We are now ready to define our main forcing notion, $\PP_E.$
\begin{definition}
$p\in \PP_E$ iff $p=  \langle f^p, A^p \rangle$ where

$(1)$ $f^p \in \PP^*_E,$

$(2)$ $A^p$ is an $E(f^p)$-tree.
\end{definition}
\begin{definition}
Let $p, q\in \PP_E.$ Then $p \leq^* q$ ($p$ is a Prikry extension of $q$) iff:

$(1)$ $f^p \leq^*_{\PP^*_E} f^q,$

$(2)$ $A^p \upharpoonright \dom(f^q) \subseteq A^q.$
\end{definition}
\begin{definition}
Let $f\in \PP^*_E, \nu \in OB(f)$ and suppose $\nu(\k) > \max(f(\k)).$ Then $f_{ \langle \nu \rangle}\in \PP^*_E$ has the same domain as $f$ and
\begin{center}
 $f_{ \langle \nu \rangle}(\a) = \left\{ \begin{array}{l}
       f(\a)^{\frown} \langle \nu(\a) \rangle  \hspace{1.1cm} \text{ if } \a\in \dom(\nu), \nu(\a) > \max(f(\a)),\\
       f(\a)  \hspace{2.5cm} \text{Otherwise}.
 \end{array} \right.$
\end{center}
Given $\langle \nu_0, \dots, \nu_{n-1} \rangle \in OB(f)^n$ such that $\nu_0(\k) > \max(f(\k))$ and $v_0 < \dots < \nu_{n-1},$ define $f_{\langle \nu_0, \dots, \nu_{n-1} \rangle}$ by recursion as
\begin{center}
$f_{\langle \nu_0, \dots, \nu_{n-1} \rangle}=(f_{\langle \nu_0, \dots, \nu_{n-2} \rangle})_{\langle \nu_{n-1} \rangle}.$
\end{center}
Let $p\in \PP_E,$ and suppose $\langle \nu_0, \dots, \nu_{n-1} \rangle \in A^p$ is such that $\nu_0(\k) > \max(f^p(\k))$ and $v_0 < \dots < \nu_{n-1}.$ Then
\begin{center}
$p_{\langle \nu_0, \dots, \nu_{n-1} \rangle}=\langle f^p_{\langle \nu_0, \dots, \nu_{n-1} \rangle}, A^p_{\langle \nu_0, \dots, \nu_{n-1} \rangle} \rangle.$
\end{center}
\end{definition}
\begin{remark}
Whenever the notation $\langle \nu_0, \dots, \nu_{n-1} \rangle$ is used, where $\nu_0, \dots, \nu_{n-1} \in OB(f),$ it is implicitly assumed $\nu_0(\k) > \max(f(\k))$ and $v_0 < \dots < \nu_{n-1}.$
\end{remark}
\begin{definition}
Let $p, q\in \PP_E.$ Then
\begin{center}
$p \leq q \Leftrightarrow \exists \langle \nu_0, \dots, \nu_{n-1} \rangle \in A^q$ $(p \leq^* q_{\langle \nu_0, \dots, \nu_{n-1} \rangle}).$
\end{center}
\end{definition}
Let's state the main properties of the forcing notion $\PP_E.$ The proof can be found in \cite{mer4}.
\begin{theorem}
Let $\text{G}$ be $\PP_{E}$-generic over $\text{V}$. Then
\begin{enumerate}
\item $\langle \PP_E, \leq \rangle$ satisfies the $\l^+-c.c.,$

\item   $\langle \PP_E, \leq, \leq^* \rangle$ satisfies the Prikry property,

\item  $\langle \PP_E, \leq^* \rangle$ is $\k$-closed,

\item $cf^{\text{V[G]}}(\k)=\omega,$

\item All $\text{V}$-cardinals in the interval $(\k, \lambda)$ are collapsed,

\item $\lambda$ is preserved in $\text{V[G]}.$
\end{enumerate}
\end{theorem}
It follow that $\text{V}$ and $\text{V[G]}$ have the same bounded subsets of $\k,$ $cf^{\text{V[G]}}(\k)=\omega,$ and $(\k^+)^{\text{V[G]}}=\lambda.$

It is possible to show that the forcing notion $\PP_E$ has enough homogeneity properties to guarantee that $\text{HOD}^{\text{V[G]}} \subseteq \text{V},$ from which we can conclude that $(\k^+)^{\text{V[G]}}=\lambda$ is inaccessible in $\text{HOD}^{\text{V[G]}}.$ However, the forcing we define for our main theorem does not have this homogeneity property, and to prove the main theorem, we need some extra work to find some intermediate model, which is a cardinal preserving extension of the universe $\text{V}$, so that the tail construction is homogeneous. To motivate that construction,
we prove the fact $(\k^+)^{\text{V[G]}}=\lambda$ is inaccessible in $\text{HOD}^{\text{V[G]}}$  in a more complicated way, which is similar to the proof of the main theorem, and so can give us some motivation for that construction.

\subsection{Projected supercompact extender based Prikry forcing}
In this subsection we define our projected forcing.
\begin{definition}
$\PP^{*, \pi}_E$ consists of all functions $f: d \rightarrow \lambda^{<\omega}$, where $d\in [j(\lambda)]^{<\lambda}$ and $\k, |d| \in  d,$ such that
\begin{enumerate}
\item For some $k<\omega,$
$f(\k)= \langle  f_0(\k), \dots, f_{k-1}(\k)               \rangle \subseteq \k $ is a finite increasing subsequence
 of $\k.$

\item For all $\k< \a\in d, f(\a)=\langle 0 \rangle.$
\end{enumerate}
For $f, g\in \PP^{*, \pi}_E,$
\begin{center}
$f \leq^{*, \pi}_{\PP^{*, \pi}_E} g \Leftrightarrow f \supseteq g.$
\end{center}
\end{definition}
\begin{remark}
$\langle \PP^{*, \pi}_E,    \leq^{*, \pi}_{\PP^{*, \pi}_E}    \rangle$ is the trivial forcing notion.
\end{remark}
\begin{lemma}
$ \PP^{*, \pi}_E$ satisfies the $\k^+$-c.c.
\end{lemma}
\begin{proof}
This follows from the fact that for $f, g \in  \PP^{*, \pi}_E,$ if $f(\k)=g(\k),$ then $f$ and $g$ are compatible, and there are only $\k$ possibilities for the choice of $f(\k)$'s.
\end{proof}
\begin{definition}
$\PP^\pi_E$ consists of all pair $p= \langle f^p, A^p \rangle$, where

$(1)$ $f^p \in \PP^{*, \pi}_E,$

$(2)$ $A^p$ is an $E(f^p)$-tree.
\end{definition}
\begin{definition}
Let $p, q\in \PP^\pi_E.$ Then $p \leq^{*,\pi} q$ ($p$ is a Prikry extension of $q$) iff:

$(1)$ $f^p \leq^{*,\pi}_{\PP^{*,\pi}_E} f^q,$

$(2)$ $A^p \upharpoonright \dom(f^q) \subseteq A^q.$
\end{definition}
\begin{definition}
Let $f\in \PP^{*,\pi}_E, \nu \in OB(f)$ and suppose $\nu(\k) > \max(f(\k)).$ Then $f_{ \langle \nu \rangle}\in \PP^{*,\pi}_E$ has the same domain as $f$ and
\begin{center}
 $f_{ \langle \nu \rangle}(\a) = \left\{ \begin{array}{l}
       f(\k)^{\frown} \langle \nu(\k) \rangle  \hspace{1.1cm} \text{ if } \a=\k,\\
       \langle 0 \rangle  \hspace{2.7cm} \text{}  Otherwise.
 \end{array} \right.$
\end{center}
Given $\langle \nu_0, \dots, \nu_{n-1} \rangle \in OB(f)^n$ such that $\nu_0(\k) > \max(f(\k))$ and $v_0 < \dots < \nu_{n-1},$ define $f_{\langle \nu_0, \dots, \nu_{n-1} \rangle}$ by recursion as
\begin{center}
$f_{\langle \nu_0, \dots, \nu_{n-1} \rangle}=(f_{\langle \nu_0, \dots, \nu_{n-2} \rangle})_{\langle \nu_{n-1} \rangle}.$
\end{center}
Let $p\in \PP^\pi_E,$ and suppose $\langle \nu_0, \dots, \nu_{n-1} \rangle \in A^p$ is such that $\nu_0(\k) > \max(f^p(\k))$ and $v_0 < \dots < \nu_{n-1}.$ Then
\begin{center}
$p_{\langle \nu_0, \dots, \nu_{n-1} \rangle}=\langle f^p_{\langle \nu_0, \dots, \nu_{n-1} \rangle}, A^p_{\langle \nu_0, \dots, \nu_{n-1} \rangle} \rangle.$
\end{center}
\end{definition}
\begin{definition}
Let $p, q\in \PP^\pi_E.$ Then
\begin{center}
$p \leq^\pi q \Leftrightarrow \exists \langle \nu_0, \dots, \nu_{n-1} \rangle \in A^q$ $(p \leq^{*,\pi} q_{\langle \nu_0, \dots, \nu_{n-1} \rangle}).$
\end{center}
\end{definition}

\subsection{Projecting $\PP_E$ onto $\PP^{\pi}_E$ }
In this subsection, we show that there is a projection from $\PP_E$ into $\PP^{\pi}_E$. For $f\in \PP^*_E,$ let $f^\pi$ be defined as follows: $\dom(f^\pi)=\dom(f),$ and for $\a\in \dom(f^\pi)$
\begin{center}
 $f^\pi(\a) = \left\{ \begin{array}{l}
       f(\a)   \hspace{1.1cm} \text{ if } \a=\k,\\
       \langle 0 \rangle  \hspace{1.45cm} \text{}  Otherwise.
 \end{array} \right.$
\end{center}
It is clear that $f^\pi \in \PP^{*, \pi}_E.$ Now define $\pi: \PP_E \rightarrow \PP^\pi_E$ by
\begin{center}
$\pi(\langle f, A \rangle) = \langle f^\pi, A    \rangle.$
\end{center}
$\pi$ is well-defined, and it is clear that it is order preserving with respect to both $\leq$ and $\leq^*$ relations, in the sense that
\begin{center}
$p \leq^* q \Rightarrow \pi(p) \leq^{*,\pi} \pi(q),$
\end{center}
and
\begin{center}
$p \leq q \Rightarrow \pi(p) \leq^\pi \pi(q).$
\end{center}
\begin{theorem}
$\pi$ is a projection of Prikry type.
\end{theorem}
\begin{proof}
Let $p=\langle f, A \rangle \in \PP_E$, $p^*=\langle f^*, A^* \rangle \in \PP^\pi_E$ and $p^* \leq^\pi \pi(p)=\langle f^\pi, A \rangle.$
Let $\langle \nu_0, \dots, \nu_{n-1} \rangle \in A$ be such that $p^* \leq^{*,\pi} \pi(p)_{\langle \nu_0, \dots, \nu_{n-1} \rangle}.$ Let $q= \langle g, B  \rangle$ where
\begin{itemize}
\item $g=f_{\langle \nu_0, \dots, \nu_{n-1} \rangle} \cup f^* \upharpoonright (\dom(f^*)\setminus \dom(f)),$
\item $B= A^*.$
\end{itemize}
Clearly $q\leq p$ (as  $q \leq^* p_{\langle \nu_0, \dots, \nu_{n-1} \rangle}$) and  $\pi(q)=\langle g^\pi, B \rangle \leq^{*, \pi} p^*,$ which gives the result.
\end{proof}

\subsection{More on $\PP^\pi_E$}
Let's state the main properties of the forcing notion $\PP^\pi_E.$
\begin{lemma}
$\PP^\pi_E$ satisfies the $\k^+$-c.c.
\end{lemma}
\begin{proof}
This follows from Lemma 3.19, and the fact that for $p, q \in \PP^\pi_E,$ if $f^p, f^q$ are compatible in $\PP^{*, \pi}_E,$ then $p, q$ are compatible in $\PP^\pi_E.$
\end{proof}
\begin{theorem}
Let $H$ be $\PP^\pi_E$-generic over $\text{V}$. Then
\begin{enumerate}

\item  $\langle \PP^\pi_E, \leq^* \rangle$ is $\k$-closed,
\item $cf^{V[H]}(\k)=\omega,$
\item   $\langle \PP^\pi_E, \leq, \leq^* \rangle$ satisfies the Prikry property.
\end{enumerate}
\end{theorem}
\begin{proof}
$(1)$ and $(2)$ are clear, and $(3)$ follows from Lemma 2.6 and Theorem 3.24.
\end{proof}
It follows that $\text{V}$ and $\text{V[H]}$ have the same cardinals and the same bounded subsets of $\k,$
and $\lambda$ remains an inaccessible cardinal in $\text{V[H]}.$

\subsection{Homogeneity of the quotient forcing}
Let $\text{G}$ be $\PP_E$-generic over $\text{V}$, and let $G^\pi$ be the filter generated by $\pi[\text{G}].$ It follows from Theorem 2.27 that $\text{G}^\pi$ is $\PP^\pi_E$-generic over $\text{V}$ and that $\text{V} \subseteq \text{V[G}^\pi] \subseteq \text{V[G]}.$ By standard forcing theorems, $\text{V[G]}$ is itself a forcing extension of $\text{V[G}^\pi].$ We show that this forcing has enough homogeneity properties, which guarantees that $\text{HOD}^{\text{V[G]}} \subseteq \text{V[G}^\pi].$

For a forcing notion $\PP$ and a condition $p\in \PP,$ set $\PP \downarrow p =\{q\in \PP: q \leq p   \}$ consists of all extensions of $p$ in $\PP.$ The homogeneity of our quotient forcing follows from the next theorem.
\begin{lemma} (Homogeneity lemma)
For all $p, q\in \PP_E,$ if $\pi(p)=\pi(q)$, then
\begin{center}
$\PP_E \downarrow p \simeq \PP_E \downarrow q$.
\end{center}
In particular, if $p\Vdash$``$\phi(\a, \vec{\gamma})$'', where $\a, \vec{\gamma}$ are ordinals, then it is not the case that $q\Vdash$``$\neg\phi(\a, \vec{\gamma})$''.
\end{lemma}
\begin{remark}
In fact, it suffices $\pi(p)$ and $\pi(q)$ to be compatible in the $\leq^{*,\pi}$-ordering relation.
\end{remark}
\begin{proof}
Since $\pi(p)=\pi(q)$, we have $\dom(f^p)=\dom(f^q)=\Delta$, $f^p(\k) =f^q(\k)$ and $A^p=A^q=A.$ Let $p^* \leq p,$ and find $\langle \nu_0, \dots, \nu_{n-1} \rangle \in A$ such that $p^* \leq^* p_{\langle \nu_0, \dots, \nu_{n-1} \rangle}.$ Let $\Phi(f^{p^*}): \dom(p^*) \rightarrow \lambda^{<\omega}$ be defined as follows:
\begin{center}
 $\Phi(f^{p^*})(\a) = \left\{ \begin{array}{l}
       f^q_{\langle \nu_0, \dots, \nu_{n-1} \rangle}(\a)  \hspace{1.1cm} \text{ if } \a\in \Delta,\\
       p^*(\a)  \hspace{2.7cm} \text{}  \a\in \dom(p^*)\setminus \Delta
 \end{array} \right.$
\end{center}
It is clear that $\Phi(f^{p^*})\in \PP^{*}_E$ and that $\Phi(f^{p^*}) \leq^*_{\PP^*_E} f^q.$
 Finally set
\begin{center}
$\Phi(p^*)=\langle \Phi(f^{p^*}), A^{p^*} \rangle.$
\end{center}
Trivially $\Phi(p^*) \leq q,$ and so $ \Phi(f^{p^*}) \in \PP_E \downarrow q.$ It is also easily seen that
\begin{center}
$\Phi: \PP_E \downarrow p \rightarrow \PP_E \downarrow q$
\end{center}
is an isomorphism. Let us just define its converse. So let  $q^* \leq q,$ and find $\langle \gamma_0, \dots, \gamma_{m-1} \rangle \in A$ such that $q^* \leq^* q_{\langle \gamma_0, \dots, \gamma_{m-1} \rangle}.$ Let $\Psi(f^{q^*}): \dom(q^*) \rightarrow \lambda^{<\omega}$ be defined as follows:
\begin{center}
 $\Psi(f^{q^*})(\a) = \left\{ \begin{array}{l}
       f^p_{\langle \gamma_0, \dots, \gamma_{m-1} \rangle}(\a)  \hspace{1.1cm} \text{ if } \a\in \Delta,\\
       q^*(\a)  \hspace{2.7cm} \text{}  \a\in \dom(q^*)\setminus \Delta
 \end{array} \right.$
\end{center}
Then set
\begin{center}
$\Psi(q^*)=\langle \Psi(f^{q^*}), A^{q^*} \rangle.$
\end{center}
Then we have
\begin{center}
$\Psi: \PP_E \downarrow q \rightarrow \PP_E \downarrow p,$
\end{center}
and it is easily seen to be the converse of $\Phi.$
\end{proof}

\begin{theorem}
$G$ be $\PP_E$-generic over $V$, and let $G^\pi$ be the filter generated by $\pi[G].$ Then
$HOD^{V[G]} \subseteq V[G^\pi].$
\end{theorem}
\begin{proof}
By the above homogeneity result.
\end{proof}

\subsection{Completing the proof of Theorem 3.1.} Finally we are ready to complete the proof of Theorem 3.1. Let $\PP_E$ and $\PP^\pi_E$ be as above, and let $G$ be $\PP_E$-generic over $V$. Let $G^\pi$ be the filter generated by $\pi[G].$ It follows from Lemma 2.2 and Theorem 3.24 that  $G^\pi$ is $\PP^\pi_E$-generic over $V$ and $V[G^\pi] \subseteq V[G].$  By Theorem 3.26, $\lambda$ remains inaccessible in $V[G^\pi]$. By Theorem 3.29, $HOD^{V[G]} \subseteq V[G^\pi],$ hence $\lambda$ remains inaccessible in  $HOD^{V[G]}.$

\section{All unocountable regular cardinals can be inaccessible in $\text{HOD}$}
In this section we give a proof of Theorem 1.3. Our strategy of the proof is similar to \cite{golshani}, but here we deal with a different  forcing notion.  As the proof of the theorem is long and complicated, we first give an idea of the proof, and then go into the details of the proof. First note that by the results from \cite{c-f-g} and \cite{friedman} , it suffices to prove the following:
\begin{theorem}
 Assume $GCH$ holds, $\kappa$ is a supercompact cardinal and $\l$ is an inaccessible cardinal above $\k$.  Then there is a generic extension $V[G]$ of $V$ in which:

$(1)$ $\kappa$ remains inaccessible,

$(2)$  There exists a club $C=\{\k_\xi:\xi<\k   \}$ of $\k$ consisting of $V$-measurable cardinals,

$(3)$  For all limit $\xi<\k, \k_\xi^+$ is inaccessible in $HOD^{V[G]}.$
\end{theorem}
So we start with  $GCH+$ a suitable  Mitchell increasing sequence of extenders  $\bar{E}=\langle E_\xi: \xi < o(\bar{E}) \rangle$  on $\k+$  $\l$ is the least inaccessible above $\k$. We define a forcing notion $\PP_{\bar{E}}$, called the ``supercompact extender based Radin forcing'', due to Merimovich \cite{mer5}, which has the following properties:
\begin{enumerate}
 \item $\PP_{\bar{E}}$ preserves the inaccessibility of $\k$,

 \item It collapses all cardinals in $(\k, \l),$ and preserves all cardinals $\geq \l,$ so $\k^+=\l$ in the extension by $\PP_{\bar{E}}$,

 \item  It adds a club $C=\{\k_\xi:\xi<\k  \}$ of $\k$ consisting of $V$-measurable cardinals,

\item If $\xi<\k$ is a limit ordinal, and $\l_\xi$ is the least inaccessible above $\k_\xi,$ then for each regular $\mu\in (\k_\xi, \l_\xi),$ there exists a cofinal sequence into $\mu$ of order type $\leq \k_\xi,$ in particular all cardinals $\mu \in (\k_\xi, \l_\xi)$ are collapsed.

 \item The forcing preserves $\l_\xi,$ so for limit $\xi<\k, \k_\xi^+=\l_\xi$ in the extension by $\PP_{\bar{E}}$. All other cardinals $<\k$ are preserved.

\end{enumerate}
 Let $G$ be $\PP_{\bar{E}}$-generic over $V$. We will show that
$V[G]$ is as required. To this end, we define a new forcing notion $\PP_{\bar{E}}^\pi$, called the ``projected supercompact extender based Radin forcing'', and a Prikry type projection $\pi: \PP_{\bar{E}} \rightarrow \PP_{\bar{E}}^\pi$. The forcing notion $\PP_{\bar{E}}^\pi$  adds the club $C$, and does not collapse any cardinals.
Furthermore, the resulting quotient forcing has enough homogeneity properties to guarantee that $HOD^{V[G]} \subseteq V[G^\pi],$ where $G^\pi$ is the filter generated by $\pi[G].$ From this results it  follows that for all limit ordinals $\xi<\k, (\k_\xi^+)^{V[G]}=\lambda_\xi$ is inaccessible in $HOD^{V[G]}$, which will give us a proof of Theorem 4.1 and hence of Theorem 1.3.

The rest of this section is devoted to the proof of Theorem 4.1. We proceed the proof in the same way as in the proof of Theorem 3.1.

\subsection{Supercompact extender based Radin forcing}
``supercompact extender based Radin forcing'' was defined by Merimovich \cite{mer5}. Our proof of Theorem 1.4. is based on this forcing notion, so we give the basic facts about it
 and present some of its main properties. All of the results of this section are due to Merimovich.  We assume that $GCH$ holds in the ground model, $\k$ is a supercompact cardinal and $\lambda$ is the least  strongly inaccessible above $\k$.
\begin{definition}
Assume $j: V \rightarrow M$ is an elementary embedding with  $crit(j)=\k, j(\k)>\lambda$ such that $M \supseteq$
$^{<\lambda}M$.

$(1)$ For each $\a$ let $\l_\a$ be minimal $\eta$ such that $\a < j(\eta).$

$(2)$ The  generators of the embedding $j$, $\mathfrak{g}(j)=\{\k_\xi: \xi\in ON  \},$ are defined by induction by
\begin{center}
$\k_\xi=\min\{\a\in ON: \forall \zeta < \xi$ $\forall \eta \in ON$ $\forall f: \eta \rightarrow ON$ $(j(f)(\k_\zeta) \neq \a ) \}.$
\end{center}
\end{definition}
If $\mathfrak{g}(j)$ is a set, then we can code $j$ by an extender $E= \langle E(\a): \a \in \mathfrak{g}(j) \rangle,$ where for each $\a \in \mathfrak{g}(j)$, $E(\a)$ is a measure on $\l_\a$ defined by
\begin{center}
$\forall A \subseteq \l_\a $ $(A \in E(\a)) \Leftrightarrow \a \in j(A).$
\end{center}
\begin{remark}
In this paper, we only deal with embeddings having their generators below $j(\lambda)$ (and hence a set), and consider the natural elementary embedding $j_E: V \rightarrow Ult(V, E).$ We may further assume that $j=j_E.$
\end{remark}

Assume $\bar{E}=\langle E_\xi: \xi < o(\bar{E})$ is a sequence of extenders on $\k$ such that:
\begin{enumerate}
\item $\bar{E}$ is Mitchell increasing, i.e., for each $\xi< o(\bar{E})$ $\langle E_\zeta: \zeta < \xi \rangle \in E_\xi,$
\item If $j_{E_\xi}: V \rightarrow M_\xi \simeq Ult(V, E_\xi)$ is the corresponding elementary embedding, then
\begin{enumerate}
\item [(2-1)] $crit(j_{E_\xi})=\k$ and $j_{E_\xi}(\k) \geq \lambda,$
\item [(2-2)] $\lambda$ is minimal such that $M_\xi \nsupseteq$ $^{\lambda}M_\xi$ (and hence $M_\xi \supseteq ^{<\lambda}M_\xi),$
\item [(2-3)] $\forall \xi < o(\bar{E})$ $\mathfrak{g}(j_{E_\xi})  \subseteq \sup j_{E_\xi}[\lambda].$
\end{enumerate}
\end{enumerate}
Note that for $\xi_1 < \xi_2 < o(\bar{E}),$ $j_{E_{\xi_1}}(\lambda) < j_{E_{\xi_2}}(\lambda).$ Let $\l \leq \epsilon \leq  \sup\{j_{E_\xi}(\lambda): \xi < o(\bar{E})  \}.$
\begin{definition}
An extender sequence $\bar{\nu}$ has the form $ \langle \tau, e_0, \dots, e_\xi, \dots  \rangle_{\xi<\mu},$ where $\langle e_\xi: \xi < \mu \rangle$ is a Mitchell increasing sequence of extenders with identical critical points and closure points, and $crit(e_0) \leq \tau < j_{e_0}(\a),$ where $\a$ is the closure point of $M_{e_0}.$ The order of the extender sequence $\bar{\nu}$ is $\mu,$ which we denote by $o(\bar{\nu})=\mu.$ We write $\bar{\nu}_0$ for $\tau$ and naturally $\bar{\nu}_{1+\xi}$ for $e_\xi.$
\end{definition}

Note that formally the Mitchell order function $o(...)$ is defined on different type of objects. The first object is of the form $\langle E_\xi: \xi<\mu \rangle,$ and the second is of the form $\langle \tau, e_0, \dots, e_\xi, \dots  \rangle_{\xi<\mu}.$ In either case only the extenders are considered, thus there is no confusion.

\begin{definition}
The set $\mathfrak{D}$ is a base set used in the domain of functions. For each $\k \leq \a < \sup\{j_{E_\xi}(\lambda): \xi < o(\bar{E})  \}, \a \notin j[\gl],$ define
\begin{center}
$\bar{\a}= \langle \a \rangle ^{\frown} \langle E_\zeta: \zeta < o(\bar{E}), \a < j_{E_\zeta}(\lambda)      \rangle.$
\end{center}
Then define
\begin{center}
$\mathfrak{D}=\{\bar{\a}: \k \leq \a < \epsilon   \}$.
\end{center}
On  $\mathfrak{D}$ the order $<$ is defined by $\bar{\a} < \bar{\beta} \Leftrightarrow \a < \beta.$ The set $\mathfrak{R}$ is used as the base for range of functions
\begin{center}
$\mathfrak{R}=\{ \bar{\nu}\in V_\lambda: \bar{\nu}$ is an extender sequence$ \}.$
\end{center}
On $\mathfrak{R}$ the order $<$ is defined by $\bar{\nu} < \bar{\mu} \Leftrightarrow \bar{\nu}_0 < \bar{\mu}_0.$
 For some technical reasons appearing in the next subsection, we assume $\langle \rangle \in \mathfrak{R}.$
\end{definition}
\begin{definition}
Assume $d\in P_\lambda(\mathfrak{D}).$ Then $\nu\in OB(d)$ iff;
\begin{enumerate}
\item $\nu: \dom(\nu) \rightarrow \mathfrak{R},$
\item $\bar{\k}\in \dom(\nu) \subseteq d \cup j[\lambda],$
\item $|\nu|< \nu(\bar{\k})_0,$
\item $\forall \a< \l$ $(j(\a) \in \dom(\nu) \Rightarrow \nu(j(\a))=\langle \a \rangle),$
\item $\forall \bar{\a}\in \dom(\nu)\setminus j[\lambda]$ $(o(\nu(\bar{\a}))< o(\bar{\a}) )$,
\item For each $\bar{\a}\in \dom(\nu)\setminus j[\lambda]$ such that $\bar{\a} \neq \bar{\k},$ the following is satisfied: Assume
\begin{center}
$\nu(\bar{\k})= \langle \tau, e_0, \dots, e_\xi, \dots     \rangle_{\xi < \zeta_\k}$ (where $crit(e_0)=\tau$)
\end{center}
and
\begin{center}
$\nu(\bar{\a})= \langle \tau', e'_0, \dots, e'_\xi, \dots     \rangle_{\xi < \zeta_\a}$.
\end{center}
Then $\langle  e_{\zeta+\xi}: \xi < \zeta_\k  \rangle=\langle  e'_{\xi}: \xi < \zeta_\a  \rangle,$ where $\zeta< \zeta_\k$ is minimal such that $\tau'\in [\sup_{\zeta'<\zeta}j_{e_{\zeta'}}(\tau), j_{e_\zeta}(\sigma)),$ where $\sigma$ is the closure point of $e_\zeta.$
\item $\forall \bar{\a}, \bar{\beta} \in \dom(\nu)\setminus j[\lambda]$ $(\bar{\a} < \bar{\beta} \Rightarrow \nu(\bar{\a}) < \nu(\bar{\beta}) ).$
\end{enumerate}
\end{definition}

On $OB(d)$ the partial order $<$ is defined by $\mu < \nu$ iff either
\begin{center}
$\forall \bar{\a}\in \dom(\mu)\cap \dom(nu)$ $($ $o(\mu(\bar{\a})) > o(\nu(\bar{\a}))$ and $\mu(\bar{\a}) < \nu(\bar{\a})  )$
\end{center}
or
\begin{center}
$\dom(\mu) \subseteq \dom(\nu)$ and $\forall \bar{\a}\in \dom(\mu)$ $($ $o(\mu(\bar{\a})) \leq o(\nu(\bar{\a}))$ and $\mu(\bar{\a}) < \nu(\bar{\a})  ).$
\end{center}
\begin{definition}
Assume $d\in P_\gl(\mathfrak{D})$.
\begin{enumerate}
\item  Assume $T \subseteq OB(d)^{<\xi} (1<\xi \leq \omega)$ and let $n<\omega.$ Then
\begin{itemize}
\item $Lev_n(T)=T \cap OB(d)^{n+1},$

\item $\Suc_T(\langle \rangle) = \Lev_0(T),$

\item $\Suc_T(\langle \nu_o, \dots, \nu_{n-1}         \rangle)=\{\mu\in OB(d): \langle \nu_o, \dots, \nu_{n-1}, \mu \rangle \in T    \}.$
\end{itemize}
\item For $\langle \nu \rangle \in T,$ let
\begin{center}
 $T_{\langle \nu \rangle}=\{  \langle \nu_o, \dots, \nu_{k-1} \rangle: k<\omega, \langle \nu, \nu_o, \dots, \nu_{k-1} \rangle \in T \}$
\end{center}
 and define by recursion for $\langle \nu_o, \dots, \nu_{n-1} \rangle \in T,$
\begin{center}
$T_{\langle \nu_o, \dots, \nu_{n-1} \rangle}= (T_{\langle \nu_o, \dots, \nu_{n-2} \rangle})_{\langle \nu_{n-1} \rangle}.$
\end{center}
\item The measures $E_\xi(d)$ $(\xi < o(\bar{E}))$ on $OB(d)$ are defined as follows:
\begin{center}
$\forall X\subseteq OB(d)$ $($ $X\in E_\xi(d) \Leftrightarrow mc_\xi(d)\in j_{E_\xi}(X) $ $)$,
\end{center}
where
\begin{center}
$mc_\xi(d)=\{ \langle j_{E_\xi}(\bar{\a}), R_\xi(\bar{\a}) \rangle : \bar{\a}\in d, \a < j_{E_\xi}(\l) \},$
\end{center}
and $R_\xi$ is defined for each $\k \leq \a < \epsilon$ by
\begin{center}
$R_\xi(\a)=\langle \a \rangle ^{\frown} \langle E_{\xi'}: \xi' < \xi,    \a < j_{E_{\xi'}}(\l)                     \rangle.$
\end{center}
Also set
\begin{center}
$E(d)=\bigcap\{ E_\xi(d): \xi < o(\bar{E})            \}.$
\end{center}
\item A tree $T \subseteq OB(d)^{<\omega}$ is called a $d$-tree, if
\begin{itemize}
\item For each $\langle \nu_o, \dots, \nu_{n-1} \rangle\in T,$ we have $\nu_0 < \dots < \nu_{n-1},$
\item $\forall \langle \nu_o, \dots, \nu_{n-1} \rangle \in T, \Suc_T(\langle \nu_o, \dots, \nu_{n-1} \rangle)\in E(d).$
\end{itemize}
\item Assume $c\in P_{\k^+}(\mathfrak{D}), c \subseteq d,$ and $T$ is a tree with elements from $OB(d)$. Then the projection of $T$ to a tree with elements from $OB(c)$ is
\begin{center}
$T \upharpoonright c = \{\langle \nu_o\upharpoonright c, \dots, \nu_{n-1}\upharpoonright c \rangle: n<\omega, \langle \nu_o, \dots, \nu_{n-1} \rangle\in T                    \}$.
\end{center}
\end{enumerate}
\end{definition}

\begin{definition}
$\PP^*_{\bar{E}, \epsilon}$ consists of all functions $f: d \rightarrow$$^{<\omega}\mathfrak{R}$ such that
\begin{enumerate}
\item $\bar{\k}\in d\in P_\l(\mathfrak{D})$,
\item For each $\bar{\a}\in d, f(\bar{\a})=\langle f_0(\bar{\a}), \dots, f_{k-1}(\bar{\a}) \rangle$ is an increasing sequence in $\mathfrak{R},$
\item For each $\bar{\a}\in d$ and $i< |f(\bar{\a})|,$ $(o(f_i(\bar{\a})) < o(\bar{\a}) ),$
\item For each $\bar{\a}\in d,$ the sequence $\langle  o(f_i(\bar{\a})): i< |f(\bar{\a})|       \rangle$ is non-increasing.
\end{enumerate}
\end{definition}
\begin{definition}
For $f, g \in \PP^*_{\bar{E}, \epsilon},$ we say $f$ is an extension of $g$ $(f \leq^*_{\PP^*_{\bar{E}, \epsilon}} g)$ if $f \supseteq g.$
\end{definition}
\begin{remark}
Clearly $\PP^*_{\bar{E}, \epsilon} \simeq Add(\gl, \epsilon).$
\end{remark}
We write $OB(f), E_\xi(f), E(f), mc_\xi(f)$ and $f$-tree, for $OB(\dom(f)), E_\xi(\dom(f)),$
$E(\dom(f)),$
\\
$mc_\xi(\dom(f))$ and $\dom(f)$-tree respectively, where $f\in \PP^*_{\bar{E}, \epsilon}.$
The following lemma is clear.
\begin{lemma}
$\langle \PP^*_{\bar{E}, \epsilon},     \leq^*_{\PP^*_{\bar{E}, \epsilon}}       \rangle$ satisfies the $\gl^+-c.c.$
\end{lemma}
\begin{definition}
Assume $f\in \PP^*_{\bar{E}, \epsilon}$ and $\nu\in OB(f).$ Define $g=f_{\langle \nu \rangle}$ to be of the form $g= g_{\leftarrow}$$^{\frown}g_{\rightarrow}$ (the case $g_{\leftarrow}=\emptyset$ is allowed) where
\begin{enumerate}
\item $\dom(g_{\rightarrow})=\dom(f),$
\item For each $\bar{\a}\in \dom(g_{\rightarrow})$
\begin{center}
 $g_{\rightarrow}(\bar{\a}) = \left\{ \begin{array}{l}
       f(\bar{\a}) \upharpoonright k$$^{\frown} \langle \nu(\bar{\a}) \rangle  \hspace{1.1cm} \text{ if } \bar{\a}\in \dom(\nu), \nu(\bar{\a}) > f_{|f(\bar{\a})|-1}(\bar{\a}),\\
       f(\bar{\a})  \hspace{3cm} \text{}  Otherwise.
 \end{array} \right.$
\end{center}
where
\begin{center}
$k=\min\{l \leq |f(\bar{\a})|: \forall l \leq i < |f(\bar{\a})|, o(f_i(\bar{\a})) < o(\nu(\bar{\a}))        \}.$
\end{center}
The above value of $k$ is defined so as to ensure that $\langle o(f_i(\bar{\a})) : i<k \rangle ^{\frown} \langle   o(\nu(\bar{\a}))   \rangle$ is non-increasing.
\item $\dom(g_{\leftarrow})=\{\nu(\bar{\a}): \bar{\a}\in \dom(\nu),  o(\nu(\bar{\a})) >0         \}$,
\item For each $\bar{\a}\in \dom(\nu)$ with $o(\nu(\bar{\a})) >0$ we have
\begin{center}
$g_{\leftarrow}(\nu(\bar{\a}))= f(\bar{\a}) \upharpoonright (|f(\bar{\a})|\setminus k),$
\end{center}
where $k$ is defined as above.
\end{enumerate}
\end{definition}
\begin{definition}
$\PP_{\bar{E}, \epsilon, \rightarrow}$ consists of pairs $p= \langle f, A \rangle$ where
\begin{enumerate}
\item $f\in \PP^*_{\bar{E}, \epsilon},$
\item $A$ is an $f$-tree such that for each $\langle \nu \rangle \in A$ and each $\bar{\a}\in \dom(\nu)$
\begin{center}
$f_{|f(\bar{\a})|-1}(\bar{\a}) < \nu(\bar{\a})$
\end{center}
\end{enumerate}
We write $f^p, A^p$ and $mc_\xi(p)$ for $f, A$ and $mc_\xi(f)$, respectively.
\end{definition}
\begin{definition}
Let $p, q \in \PP_{\bar{E}, \epsilon, \rightarrow}.$ We say $p \leq^*_{\PP_{\bar{E}, \epsilon, \rightarrow}} q$ ($p$ is a Prikry extension of $q$) if
\begin{enumerate}
\item $f^p \leq^*_{\PP^*_{\bar{E}, \epsilon}} f^q,$
\item $A^p \upharpoonright \dom(f^q) \subseteq A^q$.
\end{enumerate}
\end{definition}
\begin{definition}
Assume $\langle e_i: i< n \rangle$ $(n<\omega)$ is a sequence of extenders such that $e_i \in V_{crit(e_{i+1})}.$  The product forcing notion $\PP=\prod_{i<n}\PP_{e_i}$ is defined by applying the definitions of the Prikry with extenders forcing notions coordinatewise. That is, for each $\langle p_i: i< n \rangle, \langle q_i: i< n \rangle \in \PP,$
\begin{center}
$\langle p_i: i< n \rangle \leq_\PP \langle q_i: i< n \rangle \Leftrightarrow \forall i<n, p_i \leq_{\PP_{e_1}} q_i,$
\end{center}
and
\begin{center}
$\langle p_i: i< n \rangle \leq^*_\PP \langle q_i: i< n \rangle \Leftrightarrow \forall i<n, p_i \leq^*_{\PP_{e_1}} q_i.$
\end{center}
For $p=\langle p_i: i< n \rangle\in \PP,$ we use the notation $p_{\leftarrow}=p_0^{\frown} \dots ^{\frown} p_{n-2}$ and $p_{\rightarrow}=p_{n-1}.$ Assume $\langle \nu  \rangle \in A^{p_{\rightarrow}}.$ Define the condition $p_{\langle \nu \rangle}$ recursively as follows:
\begin{center}
$p_{\langle \nu \rangle}=p_{\leftarrow}$$^{\frown}p_{\rightarrow \langle \nu \rangle}.$
\end{center}
Note that with $p_{\leftarrow}$ and $p_{\rightarrow}$ defined we have, for each $p, q \in \PP,$
\begin{center}
$p\leq q \Leftrightarrow (p_{\leftarrow} \leq q_{\leftarrow}$ and $ p_{\rightarrow} \leq q_{\rightarrow} \leq),$
\end{center}
and
\begin{center}
$p\leq^* q \Leftrightarrow (p_{\leftarrow} \leq^* q_{\leftarrow}$ and $ p_{\rightarrow} \leq^* q_{\rightarrow} \leq).$
\end{center}
\end{definition}
It is a standard fact that $\langle \PP, \leq_\PP, \leq^*_\PP     \rangle$ is a Prikry type forcing notion, provided that all forcing notions
 $\langle \PP_{e_i}, \leq_{\PP_{e_i}}, \leq^*_{\PP_{e_i}}     \rangle$ are of Prikry type. Note that as the extenders $e_i$ are disjoint, factoring of $\PP$ is easily achieved, thus a generic extension by $\PP$ can be analyzed by inspecting generic extensions by each factor $\PP_{e_i}.$
We are now ready to define the forcing notion $\PP_{\bar{E}, \epsilon}.$
\begin{definition}
A condition $p$ in the forcing notion $\PP_{\bar{E}, \epsilon}$ is of the form $p_{\leftarrow}$$^{\frown} p_{\rightarrow}$ where
\begin{enumerate}
\item $p_{\rightarrow} \in \PP_{\bar{E}, \epsilon, \rightarrow},$
\item $p_{\leftarrow} \in \prod_{i<n} \PP_{\bar{e}_i}$ $(n<\omega)$, where $\bar{e}_i$ are extender sequences such that
\begin{enumerate}
\item [(2-1)] $o(\bar{e}_i) \leq o(\bar{E}),$
\item [(2-2)] $\bar{e}_i \in V_{crit(\bar{e}_{i+1})},$
\item [(2-3)] $\langle \nu \rangle \in A^{p_{\rightarrow}} \Rightarrow \nu(\bar{\k})_0 > crit(\bar{e}_{n-1}).$
\end{enumerate}
\end{enumerate}
\end{definition}
Conditions in $\PP_{\bar{E}, \epsilon}$ have lower parts $\PP_{\bar{E}, \epsilon, \leftarrow}$ defined by
\begin{center}
$\PP_{\bar{E}, \epsilon, \leftarrow}=\{p_{\leftarrow}: p\in \PP_{\bar{E}, \epsilon}  \}.$
\end{center}
Also for $p\in \PP_{\bar{E}, \epsilon},$ we define $f^p$ recursively to be $f^{p_{\leftarrow}}$$^{\frown} f^{p_{\rightarrow}},$ and we write $f^p_{\leftarrow}$ and $f^p_{\leftarrow}$ for $f^{p_{\leftarrow}}$ and $f^{p_{\rightarrow}}$, respectively.
\begin{definition}
Let $p, q\in \PP_{\bar{E}, \epsilon}.$ Then $p \leq^*_{\PP_{\bar{E}, \epsilon}} q$ ($p$ is a Prikry extension of $q$) if:
\begin{enumerate}
\item $p_{\rightarrow} \leq^* q_{\rightarrow},$
\item $p_{\leftarrow} \leq^* q_{\leftarrow}.$
\end{enumerate}
\end{definition}
\begin{definition}
\begin{enumerate}
\item Assume $\mu, \nu \in OB(f)$ are such that $\mu < \nu,$ and for each $\bar{\a}\in \dom(\mu), o(\mu(\bar{\a})) < o(\nu(\bar{\a})).$ Then $\mu \downarrow \nu$, the reflection of $\mu$ by $\nu$ is defined by
\begin{center}
$\dom(\mu \downarrow \nu)=\{ \nu(\bar{\a}): \bar{\a}\in \dom(\mu)     \}$
\end{center}
and for each $\bar{\a}\in \dom(\mu)$
\begin{center}
$(\mu \downarrow \nu)(\nu(\bar{\a}))=\mu(\bar{\a}).$
\end{center}
If $\mu_0, \dots, \mu_n, \nu \in OB(f)$ are such that $\mu_i < \nu$ and for each  $\bar{\a}\in \dom(\mu_i), o(\mu_i(\bar{\a})) < o(\nu(\bar{\a})),$ then
$\langle \mu_0, \dots, \mu_n \rangle \downarrow \nu,$ the reflection of $\langle \mu_0, \dots, \mu_n \rangle$ by $\nu$ is defined to be
 $\langle  \mu_0 \downarrow \nu, \dots, \mu_n \downarrow \nu   \rangle.$

\item Assume $A$ is an $f$-tree and $\langle \nu \rangle \in A.$ The tree $A \downarrow \nu$ is defined as follows: $A \downarrow \nu$ consists of all $\langle \mu_0, \dots, \mu_n \rangle \downarrow \nu$ where:
\begin{enumerate}
\item $n<\omega,$
\item $\langle \mu_0, \dots, \mu_n \rangle \in A,$
\item $\forall i\leq n$ $( \mu_i < \nu, \dom(\mu_i) \subseteq \dom(\nu)$ and $\forall \bar{\a}\in \dom(\mu_i), o(\mu_i(\bar{\a})) < o(\nu(\bar{\a})) ).$
\end{enumerate}
\end{enumerate}
\end{definition}
It is easily seen that
\begin{center}
$\{\langle \nu \rangle \in A: A \downarrow \nu$ is an $f_{\langle \nu \rangle \leftarrow}$-tree$ \} \in E_1(f),$
\end{center}
and if we consider $\emptyset$ to be an $\emptyset$-tree, then
\begin{center}
$\{\langle \nu \rangle \in A: A \downarrow \nu$ is an $f_{\langle \nu \rangle \leftarrow}$-tree$ \} \in E(f).$
\end{center}
\begin{definition}
Assume $q\in \PP_{\bar{E}, \epsilon, \rightarrow}$ and $\langle \nu \rangle \in A^q.$ The condition $p\in \PP_{\bar{E}, \epsilon}$ is the one point extension of $q$ by $\langle \nu \rangle$ ($p=q_{\langle \nu \rangle}$) if it is of the form $p_{\leftarrow}$$^{\frown}p_{\rightarrow}$ where $p_{\leftarrow} \in \PP_{\bar{e}, \rightarrow}$
and $p_{\rightarrow} \in \PP_{\bar{E}, \rightarrow}$ are defined as follows:
\begin{enumerate}
\item $f^p=f^q_{\langle \nu \rangle}$,
\item $A^{p_{\rightarrow}} =A^q_{\langle \nu \rangle},$
\item $A^{p_{\leftarrow}} =A^q \downarrow \nu.$
\end{enumerate}
Define $q_{\langle \nu_0, \dots, \nu_n \rangle}$ recursively by
\begin{center}
$q_{\langle \nu_0, \dots, \nu_n \rangle}=(q_{\langle \nu_0, \dots, \nu_{n-1} \rangle})_{\langle \nu_n \rangle}$,
\end{center}
where $\langle \nu_0, \dots, \nu_{n-1} \rangle \in A^q.$
\end{definition}
\begin{definition}
Assume $p\in \PP_{\bar{E}, \epsilon}$ and $\langle \nu \rangle \in A^{p_{\rightarrow}}.$ Then
\begin{center}
$p_{\langle \nu \rangle}=p_{\leftarrow}$$^{\frown} p_{\rightarrow}\langle \nu \rangle.$
\end{center}
Define $p_{\langle \nu_0, \dots, \nu_n \rangle}$ recursively by
\begin{center}
$p_{\langle \nu_0, \dots, \nu_n \rangle}=(p_{\langle \nu_0, \dots, \nu_{n-1} \rangle})_{\langle \nu_n \rangle}$
\end{center}
\end{definition}
\begin{definition}
Let $p, q \in \PP_{\bar{E}, \epsilon}.$ Then $p \leq_{\PP_{\bar{E}, \epsilon}} q$ ($p$ is stronger than $q$) if $p= r^{\frown} s$
and there is $\langle \nu_0, \dots, \nu_{n-1} \rangle \in A^{q_{\rightarrow}}$ such that
\begin{enumerate}
\item $s \leq^*_{\PP_{\bar{E}, \epsilon}} q_{\rightarrow \langle \nu_0, \dots, \nu_{n-1} \rangle},$
\item $r \leq q_{\leftarrow}$.
\end{enumerate}
\end{definition}
\begin{definition}
Assume $p \in \PP_{\bar{E}, \epsilon}.$ Then $p_{\rightarrow} \in \PP_{\bar{E}, \epsilon},$ and we define
\begin{center}
$\PP_{\bar{E}, \epsilon} \downarrow p_{\rightarrow} =\{q\in \PP_{\bar{E}, \epsilon}: q \leq p_{\rightarrow}         \}$
\end{center}
and
\begin{center}
$\PP_{\bar{E}, \epsilon} \downarrow p_{\leftarrow} =\{r: r \leq p_{\leftarrow}  \}.$
\end{center}
\end{definition}

Let's state the main properties of our forcing notion.
\begin{lemma}
$\PP_{\bar{E}, \epsilon}$ satisfies the $\l^+$-c.c.
\end{lemma}
\begin{proof}
Assume not, and let $A \subseteq \PP_{\bar{E}, \epsilon}$ be an antichain of size $\l^+.$ We may assume without loss of generality that all $p_{\leftarrow},$ for $p\in A$
are the same (as there are only $\lambda$-many such $p_{\leftarrow}$). Note that for any $p, q \in \PP_{\bar{E}, \epsilon},$ is $f^p$ is compatible with $f^q$ in $\PP^*_{\bar{E}, \epsilon},$ then $p$ and $q$ are compatible in $\PP_{\bar{E}, \epsilon}.$ It follows that $\{f^p: p\in A   \} \subseteq \PP^*_{\bar{E}, \epsilon}$
is an antichain of size $\gl^+,$ which contradicts Lemma 4.17.
\end{proof}
The following factorization property is clear.
\begin{lemma} (Factorization lemma)
For any $p\in \PP_{\bar{E}, \epsilon},$
\begin{center}
$\PP_{\bar{E}, \epsilon} \downarrow p \simeq \PP_{\bar{E}, \epsilon} \downarrow p_{\leftarrow} \times \PP_{\bar{E}, \epsilon} \downarrow p_{\rightarrow}.$
\end{center}
\end{lemma}
The next lemmas are proved in \cite{mer5}.
\begin{lemma}
 $\langle  \PP_{\bar{E}, \epsilon}, \leq, \leq^* \rangle$ satisfies the Prikry property.
\end{lemma}
\begin{lemma}
In a $ \PP_{\bar{E}, \epsilon}$-generic extension, $\l$ is preserved.
\end{lemma}

Let $G$ be $\PP_{\bar{E}, \epsilon}$-generic over $V$, and for $\k \leq \a < \epsilon$ let
\begin{center}
$G^{\bar{\a}}=\bigcup \{ f^p_{\rightarrow}(\bar{\a}): p\in G, \bar{\a}\in \dom(f^p_{\rightarrow})          \}$
\end{center}
and
\begin{center}
$C^{\bar{\a}}=\bigcup \{ \bar{\nu}_0: \bar{\nu}\in G^{\bar{\a}}         \}.$
\end{center}
\begin{lemma}
\begin{enumerate}
\item  $C^{\bar{\k}}$ is a club of $\k,$
\item $\a \neq \beta \Rightarrow C^{\bar{\a}} \neq C^{\bar{\beta}},$
\item  Forcing with $\PP_{\bar{E}, \epsilon}$ collapses all cardinal in $(\k, \l)$ onto $\k.$
\end{enumerate}
\end{lemma}
It follows from our results that
\begin{center}
$\l=(\k^+)^{V[G]}.$
\end{center}
Also by Lemma 4.27, $2^\k \geq |\epsilon|,$ and using the $\gl$-chain condition of the forcing, we can conclude that
$(2^\k)^{V[G]} \leq (|\PP_{\bar{E}, \epsilon}|^{<\gl})^{\k}=\epsilon,$ and hence
\begin{center}
 $V[G]\models$``$2^\k=|\epsilon|$''.
\end{center}
\begin{lemma}
 If $cf(o(\bar{E})) > |\epsilon|,$ then $\k$ remains measurable in $V[G]$.
\end{lemma}
The next theorem follows from the above results, the factorization property of the forcing notion $\PP_{\bar{E}, \epsilon}$ and using some reflection arguments.
\begin{lemma}
Assume  $G$ is $\PP_{\bar{E}, \epsilon}$-generic over $V$. Let $\langle \k_\xi: \xi<\mu \rangle$ be an increasing enumeration of $C^{\bar{\k}}$, and for each $\xi<\mu,$ let $\lambda_\xi$ be the least inaccessible above $\k_\xi.$ Then
\begin{enumerate}
\item A cardinal $\eta<\k$ is collapsed in $V[G]$ iff there exists a limit ordinal
 $\xi<\mu,$ such that $\eta \in (\k_\xi, \gl_\xi),$ and then $\eta$ is collapsed to $\k_\xi,$

\item For each limit $\xi<\mu,$ $(\k_\xi^+)^{V[G]}= \gl_\xi.$
\end{enumerate}
\end{lemma}

\subsection{Projected supercompact extender based Radin forcing}
Here we define our projected forcing $\PP_{\bar{E}, \epsilon}^\pi$. The forcing construction is very similar to that of $\PP_{\bar{E}, \epsilon},$ and so we just list the main changes which are required to define  $\PP_{\bar{E}, \epsilon}^\pi$. We add the superscript $^{\pi}$ to denote it is related to the projected forcing.

\begin{definition}
$\PP^{*,\pi}_{\bar{E}, \epsilon}$ consists of all functions $f: d \rightarrow$$^{<\omega}\mathfrak{R}$ such that
\begin{enumerate}
\item $\bar{\k}\in d\in P_\l(\mathfrak{D})$,
\item $f(\bar{\k})= \langle f_0(\bar{\k}), \dots, f_{k-1}(\bar{\k})            \rangle$ is an increasing sequence in $\mathfrak{R},$
\item $\forall \bar{\k} \neq \bar{\a}\in d, f(\bar{\a})=\langle      \rangle$ (here is the place we use the extra assumption $\langle \rangle \in \mathfrak{R}$),
\item For each  $i< |f(\bar{\k})|,$ $(o(f_i(\bar{\k})) < o(\bar{\k}) ),$
\item The sequence $\langle  o(f_i(\bar{\k})): i< |f(\bar{\k})|       \rangle$ is non-increasing.
\end{enumerate}
\end{definition}
\begin{definition}
For $f, g \in \PP^{*, \pi}_{\bar{E}, \epsilon},$ we say $f$ is an extension of $g$ $(f \leq^{*,\pi}_{\PP^{*,pi}_{\bar{E}, \epsilon}} g)$ if $f \supseteq g.$
\end{definition}
\begin{remark}
 $\PP^{*,\pi}_{\bar{E}, \epsilon}$ is forcing isomorphic to the trivial forcing.
\end{remark}
\begin{lemma}
$\langle   \PP^{*,\pi}_{\bar{E}, \epsilon},  \leq^{*,\pi}_{\PP^{*,pi}_{\bar{E}, \epsilon}} \rangle$ is $\k^+-c.c.$
\end{lemma}
\begin{proof}
The result follows from the fact that for $f,g \in \PP^{*,pi}_{\bar{E}, \epsilon},$ if $f(\bar{\k})=g(\bar{\k}),$ then $f$ and $g$ are compatible, and  that
there are only $\k$ possible choices for $f(\bar{\k})$, as each  $f(\bar{\k}) \in V_\k^{<\omega}.$
\end{proof}
\begin{definition}
Assume $f\in \PP^{*,\pi}_{\bar{E}, \epsilon}$ and $\nu\in OB(f).$ Define $g=f_{\langle \nu \rangle}$ to be of the form $g= g_{\leftarrow}$$^{\frown}g_{\rightarrow}$ (the case $g_{\leftarrow}=\emptyset$ is allowed) where
\begin{enumerate}
\item $\dom(g_{\rightarrow})=\dom(f),$
\item For each $\bar{\a}\in \dom(g_{\rightarrow})$
\begin{center}
 $g_{\rightarrow}(\bar{\a}) = \left\{ \begin{array}{l}
       f(\bar{\a}) \upharpoonright k$$^{\frown} \langle \nu(\bar{\a}) \rangle  \hspace{1.1cm} \text{ if } \bar{\a}=\bar{\k}, \nu(\bar{\a}) > f_{|f(\bar{\a})|-1}(\bar{\a}),\\
       \langle \rangle  \hspace{3.5cm} \text{}  Otherwise.
 \end{array} \right.$
\end{center}
where
\begin{center}
$k=\min\{l \leq |f(\bar{\k})|: \forall l \leq i < |f(\bar{\k})|, o(f_i(\bar{\k})) < o(\nu(\bar{\k}))        \}.$
\end{center}
The above value of $k$ is defined so as to ensure that $\langle o(f_i(\bar{\k})) : i<k \rangle ^{\frown} \langle   o(\nu(\bar{\k}))   \rangle$ is non-increasing.
\item $\dom(g_{\leftarrow})=\{\nu(\bar{\a}): \bar{\a}\in \dom(\nu),  o(\nu(\bar{\a})) >0         \}$,
\item For each $\bar{\a}\in \dom(\nu)$ with $o(\nu(\bar{\a})) >0$ we have
\begin{center}
$g_{\leftarrow}(\nu(\bar{\a}))= f(\bar{\a}) \upharpoonright (|f(\bar{\a})|\setminus k),$
\end{center}
where $k$ is defined as above.
\end{enumerate}
\end{definition}
\begin{definition}
$\PP^\pi_{\bar{E}, \epsilon, \rightarrow}$ consists of pairs $p= \langle f, A \rangle$ where
\begin{enumerate}
\item $f\in \PP^{*,\pi}_{\bar{E}, \epsilon},$
\item $A$ is an $f$-tree such that for each $\langle \nu \rangle \in A$
\begin{center}
$f_{|f(\bar{\k})|-1}(\bar{\k}) < \nu(\bar{\k})$
\end{center}
\end{enumerate}
\end{definition}
\begin{definition}
Let $p, q \in \PP^\pi_{\bar{E}, \epsilon, \rightarrow}.$ We say $p \leq^{*,\pi}_{\PP_{\bar{E}, \epsilon, \rightarrow}} q$ ($p$ is a Prikry extension of $q$) if
\begin{enumerate}
\item $f^p \leq^{*,\pi}_{\PP^{*,\pi}_{\bar{E}, \epsilon}} f^q,$
\item $A^p \upharpoonright \dom(f^q) \subseteq A^q$.
\end{enumerate}
\end{definition}
We are now ready to define the forcing notion $\PP^\pi_{\bar{E}, \epsilon}.$
\begin{definition}
A condition $p$ in the forcing notion $\PP^\pi_{\bar{E}, \epsilon}$ is of the form $p_{\leftarrow}$$^{\frown} p_{\rightarrow}$ where
\begin{enumerate}
\item $p_{\rightarrow} \in \PP^\pi_{\bar{E}, \epsilon, \rightarrow},$
\item $p_{\leftarrow} \in \prod_{i<n} \PP^\pi_{\bar{e}_i}$ $(n<\omega)$, where $\bar{e}_i$ are extender sequences such that
\begin{enumerate}
\item [(2-1)] $o(\bar{e}_i) \leq o(\bar{E}),$
\item [(2-2)] $\bar{e}_i \in V_{crit(\bar{e}_{i+1})},$
\item [(2-3)] $\langle \nu \rangle \in A^{p_{\rightarrow}} \Rightarrow \nu(\bar{\k})_0 > crit(\bar{e}_{n-1}).$
\end{enumerate}
\end{enumerate}
\end{definition}
\begin{definition}
Let $p, q\in \PP^\pi_{\bar{E}, \epsilon}.$ Then $p \leq^{*,\pi}_{\PP_{\bar{E}, \epsilon}} q$ ($p$ is a Prikry extension of $q$) if:
\begin{enumerate}
\item $p_{\rightarrow} \leq^{*,\pi} q_{\rightarrow},$
\item $p_{\leftarrow} \leq^{*,\pi} q_{\leftarrow}.$
\end{enumerate}
\end{definition}
\begin{definition}
Let $p, q \in \PP^\pi_{\bar{E}, \epsilon}.$ Then $p \leq_{\PP^\pi_{\bar{E}, \epsilon}} q$ ($p$ is stronger than $q$) if $p= r^{\frown} s$
and there is $\langle \nu_0, \dots, \nu_{n-1} \rangle \in A^{q_{\rightarrow}}$ such that
\begin{enumerate}
\item $s \leq^{*,\pi}_{\PP^\pi_{\bar{E}, \epsilon}} q_{\rightarrow \langle \nu_0, \dots, \nu_{n-1} \rangle},$
\item $r \leq q_{\leftarrow}$.
\end{enumerate}
\end{definition}
In the next subsection, we produce a projection $\pi: \PP_{\bar{E}, \epsilon} \rightarrow \PP_{\bar{E}, \epsilon}^{\pi}$ which is of Prikry type, and we will use it in subsection 4.4, to prove the main properties of the projected forcing.
\subsection{Projecting $\PP_{\bar{E}, \epsilon}$ to $\PP_{\bar{E}, \epsilon}^{\pi}$}
We now produce a Prikry type projection $\pi: \PP_{\bar{E}, \epsilon} \rightarrow \PP_{\bar{E}, \epsilon}^{\pi}$ from  $\PP_{\bar{E}, \epsilon}$ into $\PP_{\bar{E}, \epsilon}^{\pi}.$ We do it is a few steps, by producing projection between earlier forcing notions and their corresponding projected versions.

We start with defining a projection from $\PP^*_{\bar{E}, \epsilon}$ to $\PP^{*,\pi}_{\bar{E}, \epsilon}.$ Let $f\in \PP^*_{\bar{E}, \epsilon}, f: d \rightarrow \mathfrak{R}.$ Let $\pi^*(f)$ be a function with the same domain as $f$, such that
\begin{center}
 $\pi^*(f)(\bar{\a}) = \left\{ \begin{array}{l}
       f(\bar{\a})   \hspace{1.4cm} \text{ if } \bar{\a}=\bar{\k},\\
       \langle  \rangle  \hspace{2.cm} \text{}  Otherwise.
 \end{array} \right.$
\end{center}
It is clear that $\pi^*(f)\in \PP^{*,\pi}_{\bar{E}, \epsilon}$, and hence
\begin{center}
 $\pi^*: \PP^*_{\bar{E}, \epsilon} \rightarrow \PP^{*,\pi}_{\bar{E}, \epsilon}$
\end{center}
is well-defined.
\begin{lemma}
$\pi^*$ is a projection from $\PP^*_{\bar{E}, \epsilon}$ to $\PP^{*,\pi}_{\bar{E}, \epsilon}.$
\end{lemma}
\begin{proof}
Let $f\in \PP^{*}_{\bar{E}, \epsilon}, g\in \PP^{*,\pi}_{\bar{E}, \epsilon}$ and $g\leq^{*,\pi}_{\PP^{*,\pi}_{\bar{E}, \epsilon}} \pi^*(f),$ which means $g \supseteq \pi^*(f).$
Let $f^* \in \PP^{*}_{\bar{E}, \epsilon}$ have the same domain as $\dom(g),$ so that:
 \begin{center}
 $f^*(\bar{\a}) = \left\{ \begin{array}{l}
       g(\bar{\a})   \hspace{1.4cm} \text{ if } \bar{\a}=\bar{\k},\\
        f(\bar{\a})  \hspace{1.4cm} \text{ if } \bar{\a}\neq \bar{\k}, \bar{\a}\in \dom(f),\\
       \langle  \rangle  \hspace{2.cm} \text{}  Otherwise.
 \end{array} \right.$
\end{center}
Then $f^*\in \PP^{*}_{\bar{E}, \epsilon}, f^* \leq^*_{\PP^*_{\bar{E}, \epsilon}} f$ and $\pi^*(f^*) \leq^{*,\pi}_{\PP^{*,\pi}_{\bar{E}, \epsilon}} g.$
\end{proof}
We now produce a projection from $\PP_{\bar{E}, \epsilon, \rightarrow}$ into $\PP^{\pi}_{\bar{E}, \epsilon, \rightarrow}.$
Let $p=\langle f, A \rangle\in \PP_{\bar{E}, \epsilon, \rightarrow},$ and define $\pi_{\rightarrow}(\langle f, A \rangle)$ to be the pair $\langle \pi^*(f), A \rangle.$
Clearly $\pi_{\rightarrow}(\langle f, A \rangle) \in \PP^{\pi}_{\bar{E}, \epsilon, \rightarrow},$ and so
\begin{center}
$\pi_{\rightarrow}: \PP_{\bar{E}, \epsilon, \rightarrow} \rightarrow \PP^{\pi}_{\bar{E}, \epsilon, \rightarrow}$
\end{center}
is well-defined.
\begin{lemma}
$\pi_{\rightarrow}$ is a projection from $\langle \PP_{\bar{E}, \epsilon, \rightarrow}, \leq^*_{\PP_{\bar{E}, \epsilon, \rightarrow}}   \rangle$ into $\langle  \PP^\pi_{\bar{E}, \epsilon, \rightarrow}, \leq^{*,\pi}_{\PP^\pi_{\bar{E}, \epsilon, \rightarrow}}       \rangle.$
\end{lemma}
\begin{proof}
Let
$p=\langle f, A \rangle\in \PP_{\bar{E}, \epsilon, \rightarrow}, q=\langle g, B \rangle \in \PP^\pi_{\bar{E}, \epsilon, \rightarrow}$ and
$q \leq^{*,\pi}_{\PP^\pi_{\bar{E}, \epsilon, \rightarrow}} \pi_{\rightarrow}(p)=\langle \pi^*(f), A \rangle.$ This means
\begin{enumerate}
\item $g \leq^{*, \pi}_{\PP^{*,\pi}_{\bar{E}, \epsilon}} \pi^*(f),$
\item $B \upharpoonright \dom(\pi^*(f)) \subseteq A.$
\end{enumerate}
By Lemma 4.45, there exists $f^* \in \PP^*_{\bar{E}, \epsilon}$ with $f^* \leq^*_{\PP_{\bar{E}, \epsilon}}$ and $\pi(f^*) \leq^{*,\pi}_{\PP^{*, \pi}_{\bar{E}, \epsilon}} g.$
Let $p^*= \langle f^*, B \rangle \in  \PP_{\bar{E}, \epsilon, \rightarrow}$ and $\pi_{\rightarrow}(p^*) \leq^{*,\pi}_{\PP^\pi_{\bar{E}, \epsilon, \rightarrow}} q.$
\end{proof}
Finally we present a projection between our main forcing notions, namely from the forcing notion $\PP_{\bar{E}, \epsilon}$ into $\PP_{\bar{E}, \epsilon}^{\pi}$, further we show that it is in fact of Prikry type.
\begin{theorem}
There is $\pi: \PP_{\bar{E}, \epsilon} \rightarrow \PP_{\bar{E}, \epsilon}^{\pi}$ which is a Prikry type projection.
\end{theorem}
\begin{proof}
We define the projection by recursion on the critical point of the extenders. The base case was dealt in section 3.
 Now  assume that for each extender sequence $\bar{e}\in V_{\k},$ we have a projection $\pi_{\bar{e}}: \PP_{\bar{e}} \rightarrow \PP^{\pi_{\bar{e}}}_{\bar{e}},$ and we define a projection from $\PP_{\bar{E}, \epsilon}$ into $\PP_{\bar{E}, \epsilon}^{\pi}$ (where $\k$ is the critical point of extenders in $\bar{E}$).

 So let $p\in \PP_{\bar{E}, \epsilon}.$ Then $p$ has the form $p_{\leftarrow}$$^{\frown} p_{\rightarrow}$ where
\begin{enumerate}
\item $p_{\rightarrow} \in \PP_{\bar{E}, \epsilon, \rightarrow},$
\item $p_{\leftarrow} \in \prod_{i<n} \PP_{\bar{e}_i}$ $(n<\omega)$, where $\bar{e}_i$ are extender sequences such that
\begin{enumerate}
\item [(2-1)] $o(\bar{e}_i) \leq o(\bar{E}),$
\item [(2-2)] $\bar{e}_i \in V_{crit(\bar{e}_{i+1})},$
\item [(2-3)] $\langle \nu \rangle \in A^{p_{\rightarrow}} \Rightarrow \nu(\bar{\k})_0 > crit(\bar{e}_{n-1}).$
\end{enumerate}
\end{enumerate}
Let $p_{\leftarrow}=\langle p_{\bar{e_i}}: i<n     \rangle\in \prod_{i<n} \PP_{\bar{e}_i}$, and set
\begin{center}
$\pi(p)=\langle \pi_{\bar{e_i}}(p_{\bar{e_i}}) : i<n      \rangle$$^{\frown} \pi_{\rightarrow}(p_{\rightarrow}).$
\end{center}
It is easily seen that $\pi$ is as required.
\end{proof}

\subsection{More on $\PP_{\bar{E}, \epsilon}^\pi$} We use the projection $\pi$ above to prove some properties  of the forcing notion $\PP_{\bar{E}, \epsilon}^\pi$.
The next lemma can be proved as in Lemma 4.23, using Lemma 4.33 (instead of Lemma 4.11).
\begin{lemma}
$\PP_{\bar{E}, \epsilon}^{\pi}$ satisfies the $\k^+-c.c.$
\end{lemma}
Using the projection $\pi$ and the results of Section 2, we can conclude that
\begin{theorem}
 $\PP_{\bar{E}, \epsilon}^\pi$ satisfies the Prikry property.
\end{theorem}
We also have the following analogue of Lemma 4.24
\begin{lemma}
(Factorization lemma for projected forcing) For any $p\in \PP^\pi_{\bar{E}, \epsilon},$
\begin{center}
$\PP^\pi_{\bar{E}, \epsilon} \downarrow p \simeq \PP^\pi_{\bar{E}, \epsilon} \downarrow p_{\leftarrow} \times \PP^\pi_{\bar{E}, \epsilon} \downarrow p_{\rightarrow}.$
\end{center}
\end{lemma}
Let $H$ be $\PP^\pi_{\bar{E}, \epsilon}$-generic over $V$, and let
\begin{center}
$H^{\bar{\k}}=\bigcup \{ f^p_{\rightarrow}(\bar{\k}): p\in H, \bar{\k}\in \dom(f^p_{\rightarrow})          \}$
\end{center}
and
\begin{center}
$D^{\bar{\k}}=\bigcup \{ \bar{\nu}_0: \bar{\nu}\in G^{\bar{\k}}         \}.$
\end{center}
Then as in Lemma 4.27, $D^{\bar{\k}}$ can be proved to be a club of $\k.$
Also by the same arguments as in the last subsection, we can prove the following.
\begin{theorem}
$\PP_{\bar{E}, \epsilon}^\pi$ preserves all cardinals.
\end{theorem}

\subsection{Homogeneity properties}
Let $G$ be $\PP_{\bar{E}, \epsilon}$-generic over $V$, and let $G^\pi$ be the filter generated by $\pi[G].$ It follows that $G^\pi$ is $\PP_{\bar{E}, \epsilon}^\pi$-generic over $V$ and that $V \subseteq V[G^\pi] \subseteq V[G].$ By standard forcing theorems, $V[G]$ is itself a forcing extension of $V[G^\pi].$ We show that this forcing has enough homogeneity properties, which guarantees that $HOD^{V[G]} \subseteq V[G^\pi].$
The homogeneity of our quotient forcing follows from the next theorem.
\begin{lemma} (Homogeneity lemma)
For all $p, q\in \PP_{\bar{E}, \epsilon},$ if $\pi(p)=\pi(q)$, then
\begin{center}
$\PP_{\bar{E}, \epsilon} \downarrow p \simeq \PP_{\bar{E}, \epsilon} \downarrow q$.
\end{center}
In particular, if $p\Vdash$``$\phi(\a, \vec{\gamma})$'', where $\a, \vec{\gamma}$ are ordinals, then it is not the case that $q\Vdash$``$\neg\phi(\a, \vec{\gamma})$''.
\end{lemma}
\begin{proof}
To prove the lemma, we first state and prove an analogous result for forcing notions $\PP^*_{\bar{E}, \epsilon}$ and $\PP_{\bar{E}, \epsilon, \rightarrow}$, and then use these results to prove the lemma by  recursion.
\begin{itemize}
\item ({\bf Homogeneity lemma for $\PP^*_{\bar{E}, \epsilon}$}): For all $f,g\in \PP^*_{\bar{E}, \epsilon},$ if $\pi^*(f)=\pi^*(g)$, then
\begin{center}
$\PP^*_{\bar{E}, \epsilon} \downarrow f \simeq \PP^*_{\bar{E}, \epsilon} \downarrow g$.
\end{center}
\begin{proof}
As $\pi^*(f)=\pi^*(g)$, we have $\dom(f)=\dom(g)=\Delta$, and $f(\bar{\k})=g(\bar{\k}).$ Define
\begin{center}
$\Phi^*: \PP^*_{\bar{E}, \epsilon} \downarrow f \rightarrow \PP^*_{\bar{E}, \epsilon} \downarrow g$
\end{center}
as follows: let $f^* \leq^*_{\PP^*_{\bar{E}, \epsilon}} f.$ Let $\Phi^*(f^*)=g \cup f^* \upharpoonright (\dom(f^*)\setminus \Delta)$. Clearly $\Phi^*(f^*) \leq^*_{\PP^*_{\bar{E}, \epsilon}} g,$ and so $\Phi^*$ is well-defined. It is clearly an isomorphism.
\end{proof}

\item ({\bf Homogeneity lemma for $\PP_{\bar{E}, \epsilon, \rightarrow}$}): For all $p, q\in \PP_{\bar{E}, \epsilon, \rightarrow},$ if $\pi_{\rightarrow}(p)=\pi_{\rightarrow}(q)$, then
\begin{center}
$\PP_{\bar{E}, \epsilon, \rightarrow } \downarrow p \simeq \PP_{\bar{E}, \epsilon, \rightarrow} \downarrow q$.
\end{center}
\begin{proof}
Let $p= \langle f, A \rangle$ and $q= \langle g, B \rangle,$ and assume that $\pi_{\rightarrow}(p)= \langle \pi^*(f), A \rangle  = \langle \pi^*(g), B \rangle  = \pi_{\rightarrow}(q)$. Define
\begin{center}
$\Phi_{\rightarrow}: \PP_{\bar{E}, \epsilon, \rightarrow} \downarrow p \rightarrow \PP_{\bar{E}, \epsilon, \rightarrow} \downarrow q$
\end{center}
as follows: $\Phi_{\rightarrow}(\langle f^*, A^* \rangle)= \langle \Phi^*(f^*), A^*  \rangle.$ $\Phi_{\rightarrow}$ is easily seen to be well-defined and an isomorphism.
\end{proof}
\end{itemize}
We are now ready to complete the proof of the homogeneity lemma 4.47. We produce the isomorphism by recursion on the critical point of the extenders. The base case was dealt in section 3.
 Now  assume that for each extender sequence $\bar{e}\in V_{\k},$ and each $p, q \in \PP_{\bar{e}}$, if $\pi_{\bar{e}}(p)=\pi_{\bar{e}}(q)$, then
 we have an isomorphism
 \begin{center}
$\Phi_{\bar{e}}: \PP_{\bar{e}}\downarrow p \simeq\PP_{\bar{e}} \downarrow q,$
 \end{center}
and we define an  isomorphism
\begin{center}
$\Phi: \PP_{\bar{E}, \epsilon} \downarrow p \simeq \PP_{\bar{E}, \epsilon} \downarrow q$,
\end{center}
where $p, q \in \PP_{\bar{E}, \epsilon}$ are such that $\pi(p)=\pi(q).$

So let $p^*\in \PP_{\bar{E}, \epsilon}, p^* \leq_{\PP_{\bar{E}, \epsilon}} p.$ Then $p^*$ has the form $p^*_{\leftarrow}$$^{\frown} p^*_{\rightarrow}$ where
\begin{enumerate}
\item $p^*_{\rightarrow} \in \PP_{\bar{E}, \epsilon, \rightarrow},$
\item $p^*_{\leftarrow} \in \prod_{i<n} \PP_{\bar{e}_i}$ $(n<\omega)$, where $\bar{e}_i$ are extender sequences such that
\begin{enumerate}
\item [(2-1)] $o(\bar{e}_i) \leq o(\bar{E}),$
\item [(2-2)] $\bar{e}_i \in V_{crit(\bar{e}_{i+1})},$
\item [(2-3)] $\langle \nu \rangle \in A^{p_{\rightarrow}} \Rightarrow \nu(\bar{\k})_0 > crit(\bar{e}_{n-1}).$
\end{enumerate}
\end{enumerate}
Let $p^*_{\leftarrow}=\langle p^*_{\bar{e_i}}: i<n     \rangle\in \prod_{i<n} \PP_{\bar{e}_i}$, and define
\begin{center}
$\Phi(p^*)= \langle  \Phi_{\bar{e_i}}(p^*_{\bar{e_i}}): i<n    \rangle^{\frown}   \Phi_{\rightarrow}(p^*_{\rightarrow}).$
\end{center}
$\Phi$ is easily seen to be as required.
\end{proof}

\subsection{Completing the proof of Theorem 1.3}
We are finally ready to complete the proof of Theorem 1.3. As it was stated at the beginning of this section, it suffices to prove Theorem 4.1. Thus fix a Mitchell increasing sequence $\bar{E}=\langle E_\xi: \xi < o(\bar{E})$  of extenders on $\k$ as in Subsection
4.1, where $cf(o(\bar{E}))> \gl,$ and consider the forcing notions $\PP_{\bar{E}}=\PP_{\bar{E}, \gl}$ and $\PP^\pi_{\bar{E}}=\PP^\pi_{\bar{E}, \gl}$ (thus we are assuming $\epsilon=\gl$). By Theorem 4.42, there exists a projection $\pi: \PP_{\bar{E}} \rightarrow \PP^\pi_{\bar{E}}.$ Let $G$ be $\PP_{\bar{E}}$-generic over $V$ and let $G^\pi$ be the filter generated by $\pi[G].$ We know that $G^\pi$ is $\PP^\pi_{\bar{E}}$-generic over $V$ and $V \subseteq V[G^\pi] \subseteq V[G].$ Note that the clubs $C^{\bar{\k}}$ (added by $G$) and $D^{\bar{\k}}$ (added by $G^\pi$) are the same. Let $\langle \k_\xi: \xi<\k \rangle$ be an increasing enumeration of them. For each limit ordinal $\xi<\k$
\begin{enumerate}
\item In $V[G],$ $(\k_\xi)^+=\gl_\xi$, where $\gl_\xi$ is the least inaccessible above $\k_\xi,$
\item In $V[G^\pi]$, all cardinals are preserved and each $\gl_\xi$ is an inaccessible cardinal,
\item $HOG^{V[G]} \subseteq V[G^\pi].$
\end{enumerate}
It follows that for each limit ordinal $\xi<\k, (\k_\xi^+)^{V[G]}$ is an inaccessible cardinal in $HOG^{V[G]}.$ The result follows immediately.

\end{document}